\documentclass[12pt]{amsart}
\usepackage{amsmath,amscd,amsbsy,amssymb,amsthm,amsfonts}
\usepackage{latexsym,url,bm}
\usepackage{cite,enumitem}
\usepackage{fullpage}
\usepackage{color}
\usepackage[colorlinks=true]{hyperref}
\usepackage{chngcntr}
\usepackage{apptools}
\usepackage{graphicx}

\newtheorem{theorem}{Theorem}[section]
\newtheorem{proposition}[theorem]{Proposition}
\newtheorem{lemma}[theorem]{Lemma}

\theoremstyle{remark}

\theoremstyle{definition}

\allowdisplaybreaks

\numberwithin{equation}{section}

\newcommand{\p}{\partial}
\newcommand{\e}{\epsilon}

\newcommand{\R}{\mathbb{R}}
\newcommand{\N}{\mathbb{N}}
\newcommand{\Z}{\mathbb{Z}}
\newcommand{\T}{\mathbb{T}}

\renewcommand{\H}{{\mathcal H }}
\newcommand{\al}{\alpha}
\newcommand{\f}{\frac}

\newcommand{\la}{\lambda}
\newcommand{\bw}{\mathbf{W}}

\newcommand{\nP}{{\mathbf P}}
\newcommand{\CalAZ}{{\mathcal{A}_0}}
\newcommand{\ASSharp}{{\mathcal{A}_{\sharp,\frac{7}{4}}^2}}

\author{Lizhe Wan}
\address{Beijing International Center for Mathematical Research, Peking University}
\curraddr{}
\email{wanlizhe@pku.edu.cn}

\author{Jiaqi Yang}
\address{School of Mathematics and Statistics, Northwestern Polytechnical University}
\curraddr{}
\email{yjqmath@nwpu.edu.cn, yjqmath@163.com}

\begin{document}

\title{Nonlinear modulational instability of two-dimensional deep hydroelastic Stokes waves}

\begin{abstract}
In this paper, we study the nonlinear modulational instability of two-dimensional hydroelastic Stokes waves in infinite depth.
We first justify a focusing cubic nonlinear Schr\"odinger (NLS) approximation result for 2D deep hydroelastic wave system in the spirit of Ifrim-Tataru \cite{ifrim2019nls}.
Then we exploit the instability mechanism of the cubic NLS to prove that the Stokes waves are nonlinearly unstable under long-wave perturbations. 
\end{abstract}

\maketitle
%\tableofcontents

\section{Introduction}
The hydroelastic wave problem concerns the interaction between elastic structures and hydrodynamic excitation, and arises in a wide range of applications, including biology, medical science, and ocean engineering (see \cite{MR2333062, MR1312617} and the references therein). 
Starting from Cosserat shell theory under Kirchhoff’s hypotheses, which accounts for both bending stresses and membrane-stretching tension, Toland \cite{MR2413099} introduced a fully nonlinear elastic model for two-dimensional hydroelastic waves that possesses a clear Hamiltonian structure; see \cite{T07} for the membrane density considered and \cite{BT11} for the existence of steady periodic traveling hydroelastic waves. 
This model was later extended to three spatial dimensions by Plotnikov and Toland \cite{MR2812947}. 
Further developments in numerical simulation, experimentation, and applications of hydroelastic waves can be found in \cite{MR2812939,MR4823877}. 
For the Cauchy problem, we refer the reader to our recent work \cite{WanY} and the references therein.

In this paper, we are concerned with the nonlinear modulational instability of two-dimensional deep pure hydroelastic waves. 
The modulational instability, which is also known as the Benjamin–Feir instability or sideband instability, was discovered by Benjamin and Feir \cite{BF67} in 1967, whereby a small long-wave perturbation of a small Stokes wave leads to exponential instability. 
This is a phenomenon whereby deviations from a periodic waveform are reinforced by the nonlinearity, leading to the generation of spectral sidebands and the eventual breakup of the waveform into a train of pulses. 
The modulational instability also exists in various dispersive equations. See Whitham \cite{W67}, Benney and Newell \cite{BN67}, Lighthill \cite{L68}, Ostrovsky and Zakharov \cite{ZO09}, Jin, Liao, and Lin in \cite{JLL19}.

In the context of water waves, however, a rigorous proof of linear modulational instability (also referred to as spectral instability) for the full Euler equations has only been established relatively recently. 
In the 1990s, Bridges and Mielke \cite{BM95} proved spectral modulational instability for finite-depth water waves linearized around a small-amplitude Stokes wave.
More recently, significant progress has been made in the infinite-depth setting. 
In a breakthrough result, Nguyen and Strauss \cite{NS24} demonstrated spectral modulational instability under long-wave perturbations (i.e., perturbations with frequencies near zero).
Berti, Maspero, and Ventura \cite{BMV22, BMV23} further advanced the field by characterizing all unstable modes in both finite and infinite depth regimes. 
 In addition, they proved in \cite{BMV24} that small-amplitude Stokes waves are linearly modulationally unstable at the critical Whitham–Benjamin depth. 
 More recently, in \cite{BMV25}, Berti, Maspero, and Ventura established the existence of the first isola of modulational instability in the deep-water case under longitudinal perturbations. 
 Building on the modulational approximation of the water wave system and the instability mechanism of the focusing cubic nonlinear Schr\"{o}dinger equation, Chen and Su \cite{MR4627323} established nonlinear instability of small-amplitude Stokes waves in infinite depth under long-wave perturbations.

Let us now introduce the mathematical formulation of two-dimensional hydroelastic waves, where the fluid domain at time $t$ is denoted by $\Omega_t \subset \mathbb{T} \times \mathbb{R}$ and defined as the region below the graph of a function $\eta: \mathbb{R}_t \times \mathbb{T}_x \to \mathbb{R}$,
\[
\Omega_t = \{ (x, y) \in \T\times\mathbb{R} : y < \eta(t, x) \}.
\]
Its free boundary, corresponding to the deformed elastic sheet, is given by
\[
\Sigma_t = \{ (x, y) \in \T \times\mathbb{R} : y = \eta(t, x) \}.
\]
In the absence of gravity, the fluid motion is governed by the following system:
\begin{equation}\label{HWE1}
    \begin{cases}
        \mathbf{u}_t + \mathbf{u} \cdot \nabla \mathbf{u} = -\nabla p & \text{in } \Omega_t, \\
\operatorname{div} \mathbf{u} = 0, \quad \operatorname{curl} \mathbf{u} = 0 & \text{in } \Omega_t, \\
\eta_t=u_2-u_1\eta_x&\text{on } \Sigma_t,\\
        p = \sigma \mathbf{E}(\eta) & \text{on } \Sigma_t, \\
        \mathbf{u}(0, x) = \mathbf{u}_0(x) & \text{in } \Omega_t.
    \end{cases}
\end{equation}
Here $\mathbf{u} = (u_1, u_2) \in \mathbb{R}^2$ is the fluid velocity, $p$ is the pressure, $\sigma$ is the coefficient of flexural rigidity, and $ \sigma\mathbf{E}(\eta)$ represents the restoring force generated by the elastic sheet, expressed as a pressure jump across the interface, where
\begin{equation}\label{Elastic}
    \mathbf{E}(\eta) =\frac{1}{\sqrt{1+\eta_x^2}} 
    \left[ \frac{1}{\sqrt{1+\eta_x^2}} 
    \left( \frac{\eta_{xx}}{(1+\eta_x^2)^{3/2}} \right)_x \right]_x
    + \frac{1}{2} \left( \frac{\eta_{xx}}{(1+\eta_x^2)^{3/2}} \right)^3.
\end{equation}
See Toland \textit{et al.}\cite{MR2413099,MR2812947}, Guyenne and P\u ar\u au \cite{MR3002503} or Groves \textit{et al.} \cite{MR4733013,MR3566507}. 

Using holomorphic coordinates, Yang (the second author) derived the two-dimensional deep hydroelastic wave equations in Section $2$ of \cite{MR4846707}; we refer the interested reader to that work for the full derivation.
The same coordinate framework was previously employed in Hunter-Ifrim-Tataru \cite{MR3535894} and Ifrim-Tataru \cite{MR3667289} to establish well-posedness for deep gravity waves and capillary water waves, respectively, and has since been applied to a broad range of other water wave problems (see \cite{MR3535894, MR3499085, MR3625189, AIT, MR4483135, MR4462478, MR4858219, MR4891579}).
Although other formulations may be applicable, we utilize holomorphic coordinates to study the nonlinear modulational instability of the 2D hydroelastic Stokes waves in this paper.

Let $\mathbf{P}:= \frac{1}{2}(\mathbf{I} - iH)$ be the holomorphic projection that selects the holomorphic portion of the complex-valued function, with $H$ being the Hilbert transform.
On the Fourier side, $\nP$ projects onto the negative frequencies and half of the zero mode.

Introduce the holomorphic position variable $W$ and the holomorphic velocity potential $Q$, both taking values in $\mathbb{C}$ and defined on $\mathbb{R}_t \times \mathbb{T}_\alpha$. Under the derivations in \cite{MR4846707}, the free-boundary irrotational Euler equations \eqref{HWE1} and \eqref{Elastic} are equivalent to the system:
\begin{equation}\label{HF14}
	\begin{cases}
&W_t+F(1+W_{\al})=0,\\
&Q_t+FQ_{\al}+\nP\left[\f{|Q_{\al}|^2}{J}\right]
-i\sigma \nP\left\{\f{1}{J^{\f12}}\f{d}{d\al}\left[\f{1}{J^{\f12}}\f{d}{d\al}\left(\f{W_{\al\al}}{J^{\f12}(1+W_{\al})}
-\f{\bar{W}_{\al\al}}{J^{\f12}(1+\bar{W}_{\al})}\right)\right]\right\}\\
&-\f{i}{2}\sigma \nP\left\{\left[\f{W_{\al\al}}{J^{\f12}(1+W_{\al})}-\f{\bar{W}_{\al\al}}{J^{\f12}(1+\bar{W}_{\al})}\right]^3\right\}=0,
		\end{cases}
	\end{equation}
     where $J := |1+W_\alpha|^2$ is the Jacobian, and $F = \nP\left[\frac{Q_\alpha - \bar Q_\alpha}{J}\right]$.

\subsection{Main results}
Before discussing the main results in this paper, we first define the function spaces for us to work with. 
Let $(w, q)$ be the linearized variables around the zero solution, then the linearized hydroelastic waves around the zero solution are given by
\begin{equation} \label{ZeroLinear}
\left\{
\begin{array}{lr}
w_t + q_\alpha = 0 &  \\
q_t  - i\sigma\partial_\alpha^4 w =0,&  
             \end{array}
\right.
\end{equation}
restricted to negative frequencies.
The system is well-posed in the product space $\dot{H}^\f32 \times L^2$.
Therefore, we define the function spaces
\begin{equation*}
 \mathcal{H}^s :=H^{s+\f32}\times H^{s}, \quad \mathcal{H} :=H^{\f32}\times L^2.
\end{equation*}
Writing $u = |D|^{\f32}w - q$, $u$ solves the linear dispersive equation
\begin{equation*}
    (i\p_t - |D|^{\f52}) u = 0.
\end{equation*}

Our first main result asserts that hydroelastic waves can be approximated by the cubic NLS equation.
Apart from its application in proving the nonlinear modulational instability of the Stokes waves later, this approximation result has independent mathematical interest.
\begin{theorem} \label{t:NLSApproximation}
Let $U_0 \in H^3$, and $U$ be the corresponding solution to the cubic NLS equation
\begin{equation} \label{LambdaNLS}
 \left(i\partial_t + \frac{15}{8}\partial_x^2 \right) U = \frac{57}{512} U|U|^2,
\end{equation}
and let $Y^\epsilon$ be defined in \eqref{YURelation}, and $T>0$.
Then there exists a constant $\epsilon_0(\|U\|_{H^3}, T)$ so that for each $0<\epsilon< \epsilon_0$, there exists a pair of solution $(W, Q)$ to the system \eqref{HF14} with $\sigma = 1$ for $t$ in the time interval $[0, T\epsilon^{-2}]$ with the following estimate:
\begin{equation} \label{WQYEpsilonDiffBd}
 \|(W - Y^\epsilon, Q + Y^\epsilon)\|_{\H}\lesssim \epsilon^{\f32}.
\end{equation}
\end{theorem}

Same reasoning as Remark 1.10 in \cite{MR4462478}, the constant $\e_0$ may be chosen as $\e_0 = e^{-CT}$ with a large universal constant $C$.
As a consequence, the maximal time $T_{max} \approx |\log \e_0|$ in this case.
We will just pick $T = |\log \e|$ in Theorem \ref{t:modulational} below.

For the pure-hydroelastic wave system \eqref{HF14}, we will prove in Theorem \ref{t:Existence} that the Stokes waves exist.
We remark that when the gravity is also taken into account, the existence of Stokes waves is considered in \cite{MR2413099}.
Our second main result in this paper asserts that the Stokes waves that we construct are nonlinearly modulationally unstable. 

\begin{theorem} \label{t:modulational}
There exists a sufficiently small constant $\epsilon_0 \in (0,1)$ such that for all $\epsilon \in (0, \epsilon_0]$, a Stokes wave solution $(W_{ST}, Q_{ST})$ with amplitude $\epsilon$ is nonlinearly modulationally unstable in the following sense:

For any $q\in \mathbb{Q}_+$ satisfying $q \geq \frac{1}{\epsilon}$ , there exists a pair of solution $(W, Q)$ to the system \eqref{HF14} satisfying the following conditions:
\begin{enumerate}
\item At time $t=0$, $(W_0, Q_0)$ are close to the Stokes waves in the sense that
\begin{equation*}
    \|(W_0, Q_0) - (W_{ST}(0,\cdot), Q_{ST}(0, \cdot))\|_{\H(q\T)} \lesssim  \epsilon^{\frac{3}{2}}.
\end{equation*}
\item The solution $(W, Q)$ exist on $t\in [0, \epsilon^{-2}|\log \e|]$.
\item  The solution $(W, Q)$ satisfy
\begin{equation*}
    \sup_{t\in [0, \epsilon^{-2}|\log \e|]}  \|(W(t,\cdot), Q(t,\cdot)) - (W_{ST}(t, \cdot ), Q_{ST}(t,\cdot)) \|_{\H(q\T)} \gtrsim   \epsilon^{\frac{1}{2}}.
\end{equation*}
\end{enumerate}
\end{theorem}

For simplicity, in the rest of this paper, we will just consider the case where $q = \frac{1}{\epsilon}$.
The proof of Theorem \ref{t:modulational} relies heavily on Theorem \ref{t:NLSApproximation}.
By rigorously showing that the hydroelastic system is well-approximated by the focusing cubic NLS equation, we can inherit the well-known instability properties of NLS equation. 
This highlights the universality of the NLS equation as a modulation model for dispersive PDEs.

The remainder of this paper is structured as follows.
Section \ref{s:Stokes} is dedicated to the existence of pure hydroelastic Stokes waves. 
We derive the Babenko equation for the wave profiles and apply the Crandall-Rabinowitz local bifurcation theorem to construct small-amplitude solutions.
Section \ref{s:Approximation} provides the proof of Theorem \ref{t:NLSApproximation}. 
We perform a detailed normal form analysis to eliminate non-resonant quadratic and cubic terms, establishing the rigorous NLS approximation on the appropriate time scale.
Section \ref{s:Modulation} combines the previous results to prove Theorem \ref{t:modulational}. 
We use the instability of the NLS equation to construct a perturbation that grows to order $\e^\f12$ within the valid time interval of the approximation.
Appendices \ref{s:bifurcation}, \ref{s:Modified}, and \ref{s:NLSInstability} provide technical background on bifurcation theory, modified energy estimates for hydroelastic waves, and the specific instability mechanism of the cubic NLS equation, respectively.

\section{Existence of the pure hydroelastic Stokes waves}  \label{s:Stokes}   
In this section, we establish the existence of the pure hydroelastic Stokes waves, and obtain its first order asymptotic expansion.
We will first derive the Babenko equation that describes the profiles of the Stokes waves.
Then we will use the Crandall-Rabinowitz local bifurcation theorem to construct the solution of the Babenko equation. 
Although it is possible to further apply the global bifurcation theorem to obtain refined properties of the Stokes waves, we will just focus on the nonlinear modulational instability of the Stokes waves, and will not proceed further.

\subsection{Derivation of the Babenko equation}    
First, we derive the Babenko equation following the approach in \cite{Rowan24}.
We note that  the Stokes waves can be characterized as a critical point of the total energy subject to the constraint of a fixed momentum.
The velocity $c$ can be viewed as the corresponding Lagrange multiplier.
We will therefore use a variational approach to derive the real-valued Babenko equation satisfied by the Stokes waves profiles.

$(W, Q)$ are holomorphic functions.
For a holomorphic function, its real part equals the Hilbert transform of its imaginary part.
In the derivation of the Babenko equation, instead of writing a complex version of the equation using $(W,Q)$, we will just consider the real version of the Babenko equation using $(\Im W, \Im Q)$.

We define the Jacobian
\begin{equation*}
    J = |1+W_\alpha|^2 = (1+|D|\Im W)^2 +(\Im W_\alpha)^2.
\end{equation*}
Here, the differential operator $|D|$ is defined by
\begin{equation*}
    |D|f(\alpha) = \partial_\alpha Hf(\alpha).
\end{equation*}
An alternative definition of operator $|D|$ is given via the Fourier transform
\begin{equation*}
    \widehat{|D|f}(\xi) = |\xi|\hat{f}(\xi).
\end{equation*}

The total energy of \eqref{HF14} is given by
\begin{equation*}
    \mathcal{E} = \frac{1}{2}\int_0^{2\pi} |D|\Im Q \cdot\Im Q \,d\alpha +  \frac{\sigma}{2}\int_0^{2\pi}   \frac{((1+|D|\Im W) \Im W_{\al \al} - \Im W_\al |D|\Im W_\al)^2}{J^\f52} \,d\al,
\end{equation*}
where the first integral corresponds to the kinematic energy of the fluid, and the second integral corresponds to the nonlinear elastic energy of the surface.
The horizontal momentum associated to \eqref{HF14} is
\begin{equation*}
    \mathcal{P} = -\int_0^{2\pi} |D|\Im Q \Im W \, d\alpha. 
\end{equation*}

Taking the functional derivatives of the energy $\mathcal{E}$ and momentum $\mathcal{P}$ with respect to $\Im Q, \Im W$,  we get two equations.
The first equation is on the functional derivatives with respect to $\Im Q$,
\begin{equation*}
\frac{\delta \mathcal{E}}{\delta \Im Q} = c \frac{\delta \mathcal{P}}{\delta \Im Q},
\end{equation*}
which can be simplified to
\begin{equation}
\Im Q =  - c\Im W. \label{BabenkoQ}
\end{equation}

The second equation is on the functional derivative with respect to $\Im W$, 
\begin{equation*}
\frac{\delta \mathcal{E}}{\delta \Im W} = c \frac{\delta \mathcal{P}}{\delta \Im W}.
\end{equation*}
By direct computation,
\begin{align*}
 \frac{\delta \mathcal{E}}{\delta \Im W}=&\sigma\left[\frac{((1+|D|\Im W) \Im W_{\al \al} - \Im W_\al |D|\Im W_\al)(1+|D|\Im W)}{J^\f52}\right]_{\al\al}\\
&+\sigma|D|\left[\frac{((1+|D|\Im W) \Im W_{\al \al} - \Im W_\al |D|\Im W_\al)\Im W_{\al}}{J^\f52}\right]_{\al}\\
&+\sigma\left[\frac{((1+|D|\Im W) \Im W_{\al \al} - \Im W_\al |D|\Im W_\al)|D|\Im W_{\al}}{J^\f52}\right]_{\al}\\
&+\sigma|D|\left[\frac{((1+|D|\Im W) \Im W_{\al \al} - \Im W_\al |D|\Im W_\al)\Im W_{\al\al}}{J^\f52}\right]\\
&+\f52\sigma\left[\frac{((1+|D|\Im W) \Im W_{\al \al} - \Im W_\al |D|\Im W_\al)^2\Im W_{\al}}{J^\f72}\right]_{\al}\\
&-\f52\sigma|D|\left[\frac{((1+|D|\Im W) \Im W_{\al \al} - \Im W_\al |D|\Im W_\al)^2(1+|D|\Im W)}{J^\f72}\right],\\
  \frac{\delta \mathcal{P}}{\delta \Im W} =& -|D|\Im Q,
\end{align*}
which leads to the second equation
\begin{align*}
-c|D|\Im Q=&\sigma\left[\frac{((1+|D|\Im W) \Im W_{\al \al} - \Im W_\al |D|\Im W_\al)(1+|D|\Im W)}{J^\f52}\right]_{\al\al}\\
&+\sigma|D|\left[\frac{((1+|D|\Im W) \Im W_{\al \al} - \Im W_\al |D|\Im W_\al)\Im W_{\al}}{J^\f52}\right]_{\al}\\
&+\sigma\left[\frac{((1+|D|\Im W) \Im W_{\al \al} - \Im W_\al |D|\Im W_\al)|D|\Im W_{\al}}{J^\f52}\right]_{\al}\\
&+\sigma|D|\left[\frac{((1+|D|\Im W) \Im W_{\al \al} - \Im W_\al |D|\Im W_\al)\Im W_{\al\al}}{J^\f52}\right]\\
&+\f52\sigma\left[\frac{((1+|D|\Im W) \Im W_{\al \al} - \Im W_\al |D|\Im W_\al)^2\Im W_{\al}}{J^\f72}\right]_{\al}\\
&-\f52\sigma|D|\left[\frac{((1+|D|\Im W) \Im W_{\al \al} - \Im W_\al |D|\Im W_\al)^2(1+|D|\Im W)}{J^\f72}\right].   
\end{align*}

Therefore by eliminating $\Im Q$ in the second equation using \eqref{BabenkoQ}, we obtain that $U: =\Im W$ solves the Babenko equation
\begin{equation} \label{e:Babenko}
  G\left(\frac{c^2}{\sigma}, U \right) = 0,  
\end{equation}
where
\begin{align*}
G&\left(\lambda, U \right) := \lambda |D|U-\left[\frac{((1+|D|U)  U_{\al \al} -U_\al |D|U_\al)(1+|D|U)}{J^\f52}\right]_{\al\al} \\
&-|D|\left[\frac{((1+|D|U) U_{\al \al} - U_\al |D|U_\al)U_{\al}}{J^\f52}\right]_{\al}-\left[\frac{((1+|D|U) U_{\al \al} - U_\al |D|U_\al)|D|U_{\al}}{J^\f52}\right]_{\al} \\
&-|D|\left[\frac{((1+|D|U) U_{\al \al} - U_\al |D|U_\al)U_{\al\al}}{J^\f52}\right]-\f52\left[\frac{((1+|D|U) U_{\al \al} -U_\al |D|U_\al)^2U_{\al}}{J^\f72}\right]_{\al} \\
&+\f52|D|\left[\frac{((1+|D|U) U_{\al \al} - U_\al |D|U_\al)^2(1+|D|U)}{J^\f72}\right] .
\end{align*}

$U = 0$ is a trivial solution of \eqref{e:Babenko}. 
If one can obtain a non-trivial solution $U$ of \eqref{e:Babenko}, then \eqref{HF14} has a Stokes wave solution
\begin{equation*}
    \left( HU(\alpha + ct) + iU(\alpha + ct),   cHU(\alpha + ct)+icU(\alpha + ct)\right).
\end{equation*}

\subsection{Local bifurcation of the Babenko equation}
Next, we use the Crandall-Rabinowitz local bifurcation theorem to construct the non-trivial solution of the Babenko equation \eqref{e:Babenko}.
We will assume without loss of generality that $U$ is an even function with mean value zero.

Since $U$ is a trivial solution of \eqref{e:Babenko}, $G(\lambda, 0) = 0$ for all $\lambda \in \R$.
Although the Babenko equation \eqref{e:Babenko} looks complicated, it turns out that $\partial_U G[(\lambda, 0)]V$ is simple.
\begin{equation*}
LV := \partial_U G[(\lambda, 0)]V =  \lambda |D|V - \p_\al^4 V.
\end{equation*}

First, we show that $L$ is a Fredholm operator of index zero.
We write $L = T + K$, where $T = - \p_\al^4$ and $K = \lambda|D|$.
$T$ is a fourth-order linear elliptic differential operator. 
In the space of functions with mean value zero, $T$ is a homeomorphism from $H^s_e(\T)$ to $H_e^{s-4}(\T)$,  where $H^s_e(\T):=\{u\in H^s(\T): \text{$u$ is even}\}$. Since the inclusion $H_e^{s-1}(\T) \hookrightarrow H_e^{s-4}(\T)$ is compact, and $K$ maps $H_e^{s}(\T) \rightarrow H_e^{s-1}(\T)$ continuously, $K$ is therefore a compact operator from $H_e^s(\T)$ to $H_e^{s-4}(\T)$.
Hence, by Lemma \ref{t:Fredholm}, $L = T+K$ is a Fredholm operator with index zero.

Next, we consider the kernel of $L$.
\begin{equation*}
  \lambda |D|V - \p_\al^4 V = 0 \Leftrightarrow  (\la |n|-|n|^4)\hat{V}_n =0, \quad\forall n\in\mathbb{Z},\quad\hat{V}_n:=\f{1}{2\pi}\int_0^{2\pi}V(\al) e^{-in\al}d\al.
\end{equation*}
If there exists $k\in \mathbb{Z}\setminus\{0\}$ such that $\la= |k|^3$, then ker $ L=\text{span}\{\cos(k\al)\}$.

Then, we verify the transversality condition.
Assuming that $\lambda_0 = |k_0|^3$ for some $k_0 \in \N^+$, we choose $\xi_0 = \cos k_0\al$, and compute
\begin{equation*}
\p_{\lambda, U}^2 G[(\lambda_0, 0)](1, \xi_0) = |D|\cos k_0 \al = k_0 \cos k_0 \al.
\end{equation*}
Note that range$(L) = \{f\in H^{s-4}_e : \langle f, \cos k_0\al \rangle_{L^2} = 0 \}$.
Since 
\begin{equation*}
 \langle k_0 \cos k_0 \al, \cos k_0\al \rangle_{L^2} = \int_0^{2\pi} k_0 \cos^2 (k_0\al) \, d\al = k_0 \pi \neq 0,    
\end{equation*}
we conclude that 
\begin{equation*}
  \p_{\lambda, U}^2 G[(\lambda_0, 0)](1, \xi_0) \notin \text{range}(L).  
\end{equation*}

Therefore using Crandall-Rabinowitz bifurcation theorem, there exists a constant $\varepsilon_0$ such that for $|\varepsilon|<\varepsilon_0$, 
\begin{equation*}
    G(\Lambda(\varepsilon), \varepsilon\chi(\varepsilon)) = 0, \quad \text{where}\quad \Lambda(0) = k_0, \quad \chi(0) = \cos (k_0\al).
\end{equation*}
$\Lambda$ and $s \mapsto s\chi(s)$ are of class $C^1$, and $\chi$ is of class $C^0$ on $(-\varepsilon, \varepsilon)$.

In the rest of this paper, for simplicity, we will simply  consider $k_0 = 1$.
Moving back to $(W, Q)$ variables in \eqref{HF14}, this means that $(W_{ST}, Q_{ST})$ are frequency-localized at around $\xi_0 = -1$.
We obtain the following result on the existence of the Stokes waves for two-dimensional deep hydroelastic waves.

\begin{theorem} \label{t:Existence}
There exists a constant $\varepsilon_0$ such that if $\frac{c^2}{\sigma} = 1 -\frac{57}{256}\varepsilon^2$, for $ |\varepsilon| < \varepsilon_0$, \eqref{HF14} has a Stokes wave solution
\begin{equation*}
  (W_{ST}(t,\al), Q_{ST}(t, \al)) = (W_{ST}(\al+ ct), Q_{ST}(\al + ct)).
\end{equation*}
Moreover, $(W_{ST}, Q_{ST})$ are frequency-localized at around $\xi_0 =-1$, in the sense that
\begin{equation*}
    (W_{ST}(t,\al), Q_{ST}(t,\al)) = \varepsilon (ie^{-i(\al+ ct)}, -cie^{-i(\al+ ct)}) + O(\varepsilon^2) = \varepsilon (ie^{-i(\al+ ct)}, -\sqrt{\sigma}ie^{-i(\al+ ct)}) + O(\varepsilon^2).
\end{equation*}
\end{theorem}
In the following sections, we assume $\sigma = 1$ for simplicity.
In addition, we will choose
\begin{equation*}
    \varepsilon = \sqrt{\frac{512}{57}}\epsilon,
\end{equation*}
so that the amplitude of the Stokes waves will be compatible with the Stokes of the cubic NLS equation \eqref{LambdaNLS} after rescaling in Section \ref{s:Modulation}.
The choice $c^2 = 1 -\frac{57}{256}\varepsilon^2$ gives $c \approx 1 -\frac{57}{512}\varepsilon^2$ or $c\approx -1 + \frac{57}{512}\varepsilon^2$.
This choice of $c$ will later be needed in the estimate \eqref{512Constant}.
We will just consider the case $c>0$, the other case can be obtained by using the time-reversal transformation $\hat{W}(t,\al)=W(-t,\al)$, $\hat{Q}(t,\al)=-Q(-t,\al)$.

\section{The NLS approximation of the hydroelastic waves} \label{s:Approximation}
In this section, we prove the first main result of this paper, that the hydroelastic wave system \eqref{HF14} can be approximated by cubic focusing NLS equation on the cubic lifespan.

We begin with a heuristic discussion for the linearized hydroelastic waves around the zero solution.
Recall that the system is \eqref{ZeroLinear}.

The dispersion relation for this linear evolution has two branches. 
Let $\tau$ be the Fourier variable in time and $\xi$ be the Fourier variable in space.
Then $\tau$ and $\xi$ satisfy 
\begin{equation*}
 \tau = \omega_{\pm}(\xi) : = \pm |\xi|^{\frac{5}{2}}, \quad \xi <0.
\end{equation*}
As discussed in \cite{ifrim2019nls}, the NLS approximation
applies to solutions which are localized on a single branch near a single frequency $\xi_0$.
In view of the frequency localization of the Stokes waves proved in Theorem \ref{t:Existence}, we choose $\xi_0 = -1$ and the $-$ sign.

Using the second-order Taylor expansion for $\omega_-(\xi)$ near $\xi_0 = -1$, we obtain the second-order approximation of the dispersion relation:
\begin{equation*}
    \tau = \omega_0(\xi) : = -1 + \frac{5}{2}(\xi+1) - \frac{15}{8}(\xi+1)^2.
\end{equation*}
The corresponding linear flow
\begin{equation*}
 (i\partial_t + \omega_0(D))y = 0
\end{equation*}
can be recast as the linear Schr\"odinger flow
\begin{equation} \label{LinearSchrodinger}
\left(i\partial_t + \frac{15}{8}\p_x^2 \right) u = 0
\end{equation}
via the transformation 
\begin{equation*}
    y(t,x) = e^{-i(t+x)}u\left(t, x+ \frac{5}{2}t \right).
\end{equation*}
Hence, for $y$ as the solution of the linear flow, $(w, q)$ can be approximated by $(w,q) \approx (y,y)$ for solutions near frequency $\xi_0 =-1$ on the negative branch.

Since $\omega_0$ is only a second-order approximation of $\omega_-$, if the frequencies of the solution are concentrated in an $\epsilon-$neighborhood of $\xi_0 = -1$, we have the difference relation 
\begin{equation*}
    |\omega_0 - \omega_+|\lesssim \epsilon^3.
\end{equation*}
The linear evolution $e^{it\omega_0(D)}$ is therefore a good approximation of $e^{it \omega_-(D)}$ on the quartic time scale $|t|\ll \epsilon^{-3}$.
For $u_0 = O(1)$, 
\begin{equation*}
    |e^{it\omega_0(D)} u_0 - e^{it \omega_-(D)} u_0| \lesssim \epsilon, \quad \text{on } |t|\lesssim \epsilon^{-2}.
\end{equation*}

Next, we take into account the nonlinear source terms of the hydroelastic waves.
Inspired by the result of the cubic lifespan of the hydroelastic waves in \cite{MR4846707}, we consider the cubic NLS equation
\begin{equation} \label{CubicNls}
 \left(i\partial_t + \frac{15}{8}\partial_x^2 \right) U = \lambda U|U|^2,
\end{equation}
for some real constant $\lambda$ which will later shown to be $\frac{57}{512}$.
We rescale the functions,
\begin{equation} \label{YURelation}
 U^{\epsilon}(t,x) : = \epsilon U(\epsilon^2 t, \epsilon x), \quad  Y^\epsilon(t,x) = e^{-i(t+x)}U^{\epsilon}\left(t, x+ \frac{5}{2}t \right).
\end{equation}
Then the new rescaled variable $Y^\epsilon$ solves the cubic nonlinear dispersive equation
\begin{equation} \label{CubicEpsilonDispersive}
 (i\partial_t + \omega_0(D))Y^\epsilon = \lambda Y^\epsilon |Y^\epsilon|^2.   
\end{equation}

The first main result in this paper Theorem \ref{t:NLSApproximation} gives the approximation result which shows that the cubic nonlinear dispersive equation \eqref{CubicEpsilonDispersive} is a good approximation of the hydroelastic waves \eqref{HF14}.
This section is devoted to the proof of this result, which is divided into three steps.
We begin with the cubic NLS equation \eqref{CubicNls}. 
In the first step, we truncate $U$ in frequency and obtain the corresponding $\tilde{Y}^\epsilon$.
For the second step, we use the normal form analysis from \cite{MR4846707} to construct a good approximate
solution to the hydroelastic waves which is close to $(\tilde{Y}^\epsilon, \tilde{Y}^\epsilon)$.
Finally, we show that the approximate solutions $(W^\epsilon, Q^\epsilon)$ can be replaced by the exact solutions $(W, Q)$.

\subsection{The NLS truncation}
In the first step, we perform the frequency truncation for the NLS equation \eqref{CubicNls}.
The approach is the same as in Section 2 of \cite{ifrim2019nls}, and we recall the results there.

Let $U$ solve the cubic NLS equation \eqref{CubicNls}, we define the frequency-truncated variable
\begin{equation*}
    \tilde{U} : = P_{\leq c\epsilon^{-1}}U,
\end{equation*}
where $P_{\leq c\epsilon^{-1}}$ is the frequency projection operator that selects frequencies $\leq c\epsilon^{-1}$ for some small constant $c$.
Then, $\tilde{U}$ solves the equation
\begin{equation*}
  \left(i\partial_t + \frac{15}{8}\partial_x^2 \right) \tilde{U} = \lambda \tilde{U}|\tilde{U}|^2 + \tilde{f},\quad \tilde{f}: = \left(i\partial_t + \frac{15}{8}\partial_x^2 \right) \tilde{U} - \lambda \tilde{U}|\tilde{U}|^2.
\end{equation*}
Then, we have the following result from Proposition $1.5$ in \cite{ifrim2019nls}.
\begin{proposition}[\hspace{1sp}\cite{ifrim2019nls}]
Let $U$ be a solution of \eqref{CubicNls} with initial data $\|U_0\|_{H^s} \lesssim 1$ for some $s\geq \frac{3}{2}$, then for $\tilde{U}$ that we define above, we have
\begin{enumerate}
\item The uniform bounds
\begin{equation*}
    \|\tilde{U}\|_{H^k} \lesssim \epsilon^{s-k}, \quad k \geq s.
\end{equation*}
\item The difference bounds
\begin{equation*}
 \|U - \tilde{U} \|_{L^2} \lesssim \epsilon^{s}.
\end{equation*}
\item The error estimates
\begin{equation*}
\|\tilde{f}\|_{L^2} \lesssim \epsilon^{s+1}.
\end{equation*}
\end{enumerate}
\end{proposition}
We further define $\tilde{Y}^\epsilon$ by
\begin{equation*}
    \tilde{Y}^\e(t,x):=e^{-i(t+x)}\epsilon\tilde{U}\left(\epsilon^2t, \e x+ \frac{5}{2}\e t \right).
\end{equation*}
We get by scaling,
\begin{equation} \label{TildeYEpsilonBound}
\|\tilde{Y}^\epsilon \|_{L^2} \lesssim \epsilon^\f12, \quad \|\tilde{Y}^\epsilon \|_{L^\infty} \lesssim \epsilon.
\end{equation}
The smallness of $c$ above
guarantees that $\tilde{Y}^{\e}$
is frequency localized in a small neighborhood of $\xi_0=-1$.
We then have
\begin{equation*}
\|(D+1)\tilde{Y}^{\e}\|_{L^2}\lesssim \epsilon^\f32.
\end{equation*}
By the above proposition, 
\[
\| \tilde{Y}^\epsilon - Y^\epsilon \|_{H^{\frac{5}{2}}} \lesssim \epsilon^{s-\frac{5}{2}}, \quad s \geq \frac{5}{2}.
\]
In addition, $\tilde{Y}^\epsilon$ is an approximate solution to \eqref{CubicEpsilonDispersive}, and it solves 
\begin{equation} \label{TildeYEpsilon}
(i \partial_t + \omega_0(D)) \tilde{Y}^\epsilon = \lambda \tilde{Y}^\epsilon |\tilde{Y}^\epsilon|^2 + g^\epsilon,    
\end{equation}
with the error
\[
g^\epsilon := (i \partial_t + \omega_0(D)) \tilde{Y}^\epsilon - \lambda \tilde{Y}^\epsilon |\tilde{Y}^\epsilon|^2
\]
satisfying the bound $\| g^\epsilon \|_{L^2} \lesssim \epsilon^{s+\frac{7}{2}}$.
Later, we will just consider $s = 3$.

\subsection{Normal form analysis and NLS approximation}
In the second step, we construct a good approximate solution $(W^\e, Q^\e)$ to the hydroelastic waves that is close to $(\tilde{Y}^\epsilon, \tilde{Y}^\epsilon)$.

It follows from the computations in \cite{MR4846707} that one can expand the nonlinear terms in \eqref{HF14}, and rewrite the system as
\begin{equation}\label{NF1}
\begin{cases}		W_t+Q_{\al}=G:=G^{(2)}+G^{(3)}+G^{(4+)}\,,\\
Q_t-iW_{\al\al\al\al}=K:=K^{(2)}+K^{(3)}+K^{(4+)}\,,
		\end{cases}
	\end{equation}
where the quadratic terms $\left(G^{(2)},K^{(2)}\right)$ are given by
	\begin{equation}\label{NF2}
		\begin{cases}
			G^{(2)}:=\nP[Q_{\al}\bar{W}_{\al}-\bar{Q}_{\al}W_{\al}]\,,\\
			K^{(2)}:=-Q^2_{\al}-\nP[|Q_{\al}|^2]-i\nP[\f52W_{\al}W_{\al\al\al\al}+5W_{\al\al}W_{\al\al\al}+\f32\bar{W}_{\al}W_{\al\al\al\al}\\
			\qquad\quad \ \ +\bar{W}_{\al\al}W_{\al\al\al}-W_{\al\al}\bar{W}_{\al\al\al}-\f32W_{\al}\bar{W}_{\al\al\al\al}],
		\end{cases}
	\end{equation}
and the cubic terms $\left(G^{(3)},K^{(3)}\right)$ read
	\begin{equation}\label{NF3}
		\begin{cases}
			G^{(3)}:=&W_{\al}(Q_{\al}W_{\al}-\nP[\bar{Q}_{\al}W_{\al}-Q_{\al}\bar{W}_{\al}])-\nP[(Q_{\al}-\bar{Q}_{\al})(4|\mathfrak{R}W_{\al}|^2-|W_{\al}|^2),
			\\
			K^{(3)}:=&Q_{\al}(Q_{\al}W_{\al}-\nP[\bar{Q}_{\al}W_{\al}-Q_{\al}\bar{W}_{\al}])+2\nP[\mathfrak{R}W_{\al\al}|Q_{\al}|^2]
			+i\nP\Big[\f{35}{8}W^2_{\al}W_{\al\al\al\al}\\
			&+\f74|W_{\al}|^2W_{\al\al\al\al}+\f{15}{8}\bar{W}^2_{\al}W_{\al\al\al\al}+\f{35}{2}W_{\al}W_{\al\al}W_{\al\al\al}
			+\f{15}{2}\bar{W}_{\al}W_{\al\al}W_{\al\al\al}\\
			&+\f{15}{4}W_{\al}\bar{W}_{\al\al}W_{\al\al\al}+\f{15}{4}\bar{W}_{\al}\bar{W}_{\al\al}W_{\al\al\al}
			+\f54W_{\al}W_{\al\al}\bar{W}_{\al\al\al}+\f54\bar{W}_{\al}W_{\al\al}\bar{W}_{\al\al\al}\\
			&+\f92W^3_{\al\al}+\f52W_{\al\al}|W_{\al\al}|^2+\bar{W}_{\al\al}|W_{\al\al}|^2
			-\f74|W_{\al}|^2\bar{W}_{\al\al\al\al}
			-\f{15}{8}W^2_{\al}\bar{W}_{\al\al\al\al}\\
			&-\f{15}{2}W_{\al}\bar{W}_{\al\al}\bar{W}_{\al\al\al}
			-\f{15}{4}\bar{W}_{\al}W_{\al\al}\bar{W}_{\al\al\al}-\f{15}{4}W_{\al}W_{\al\al}\bar{W}_{\al\al\al}\\
			&-\f54\bar{W_{\al}}\bar{W}_{\al\al}W_{\al\al\al}-\f54W_{\al}\bar{W}_{\al\al}W_{\al\al\al}
			-\f52\bar{W}_{\al\al}|W_{\al\al}|^2-W_{\al\al}|W_{\al\al}|^2\Big]\\
			&+\f{i}{2}\nP[W^3_{\al\al}-3W_{\al\al}|W_{\al\al}|^2+3\bar{W}_{\al\al}|W_{\al\al}|^2].
		\end{cases}
	\end{equation}
$\left(G^{(4+)},K^{(4+)}\right)$ are quartic and higher terms in derivatives of $W$ and $Q$.
As in \cite{MR4846707}, the quadratic terms $G^{(2)}$ and $K^{(2)}$ can be eliminated using a normal form transformation
\begin{equation}\label{NFP1-1}
\begin{cases}
\tilde{W}=W+W_{[2]}\,,\\
\tilde{Q}=Q+Q_{[2]}\,,
\end{cases}
\end{equation}
		where the bilinear forms $(W_{[2]},Q_{[2]})$ are
\begin{equation}\label{NFP1-2}
\begin{aligned}
&W_{[2]}=B^h(W,W)+C^h(Q,Q)+B^a(W,\bar{W})+C^a(Q,\bar{Q})\,,\\
&Q_{[2]}=A^h(W,Q)+A^a(W,\bar{Q})+D^a(Q,\bar{W}).
			\end{aligned}
		\end{equation}
The bilinear symbols of these bilinear forms are as follows:
		\begin{equation*}
			\begin{aligned}
			&A^h=i\f{5\xi^7-\f{5}{2}\xi^6\eta-40\xi^5\eta^2-95\xi^4\eta^3-121\xi^3\eta^4-100\xi^2\eta^5-50\xi\eta^6-\f{25}{2}\eta^7}
				{25(\xi^3+2\xi^2\eta+2\xi\eta^2+\eta^3)^2-4\xi^3\eta^3}\,,\\
				&B^h=-i(\xi+\eta)\f{\f{25}{4}(\xi^3+2\xi^2\eta+2\xi\eta^2+\eta^3)^2-2\xi^3\eta^3}{25(\xi^3+2\xi^2\eta+2\xi\eta^2+\eta^3)^2-4\xi^3\eta^3}\,,\\
				&C^h=-i\f{\f52(\xi+\eta)(\xi^3+2\xi^2\eta+2\xi\eta^2+\eta^3)}{25(\xi^3+2\xi^2\eta+2\xi\eta^2+\eta^3)^2-4\xi^3\eta^3},
			\end{aligned}
		\end{equation*}
		and
\begin{equation*}
\begin{aligned}
&A^a=i\f{10\xi^8+45\xi^7\eta+112\xi^6\eta^2+154\xi^5\eta^3+62\xi^4\eta^4-140\xi^3\eta^5-240\xi^2\eta^6-155\xi\eta^7-40\eta^8}{2\xi(25\xi^6+100\xi^5\eta+200\xi^4\eta^2+246\xi^3\eta^3+200\xi^2\eta^4+100\xi\eta^5+25\eta^6)}\,,\\
&B^a=i\f{-25\xi^8-150\xi^7\eta-420\xi^6\eta^2-722\xi^5\eta^3-822\xi^4\eta^4-620\xi^3\eta^5-293\xi^2\eta^6-76\xi\eta^7-8\eta^8}{2\xi(25\xi^6+100\xi^5\eta+200\xi^4\eta^2+246\xi^3\eta^3+200\xi^2\eta^4+100\xi\eta^5+25\eta^6)}\,,\\
&C^a=-i\f{5\xi^5+15\xi^4\eta+16\xi^3\eta^2+3\xi^2\eta^3-7\xi\eta^4-4\eta^5}
{\xi(25\xi^6+100\xi^5\eta+200\xi^4\eta^2+246\xi^3\eta^3+200\xi^2\eta^4+100\xi\eta^5+25\eta^6)}\,,\\
&D^a=-i\f{25\xi^8+125\xi^7\eta+320\xi^6\eta^2+542\xi^5\eta^3+662\xi^4\eta^4+580\xi^3\eta^5+345\xi^2\eta^6+125\xi\eta^7+20\eta^8}
{2\xi(25\xi^6+100\xi^5\eta+200\xi^4\eta^2+246\xi^3\eta^3+200\xi^2\eta^4+100\xi\eta^5+25\eta^6)}.
\end{aligned}
\end{equation*}

Using the normal form variables $(\tilde{W}, \tilde{Q})$, the system \eqref{NF1} becomes 
\begin{equation*}
\begin{cases}		\tilde{W}_t+\tilde{Q}_{\al}=\tilde{G}_r^{(3)}+ \tilde{G}_{nr}^{(3)}+\tilde{G}_{null}^{(3)}+ \tilde{G}^{(4+)}\,,\\
\tilde{Q}_t-i\tilde{W}_{\al\al\al\al}= \tilde{K}_r^{(3)}+ \tilde{K}_{nr}^{(3)}+  \tilde{K}_{null}^{(3)}+ \tilde{K}^{(4+)},
		\end{cases}
\end{equation*}
where $(\tilde{G}_r^{(3)}, \tilde{K}_r^{(3)})$ are cubic resonant terms, $(\tilde{G}_{nr}^{(3)}, \tilde{K}_{nr}^{(3)})$ are cubic non-resonant terms, $(\tilde{G}_{null}^{(3)}, \tilde{K}_{null}^{(3)})$ are cubic null terms, and $(\tilde{G}^{(4+)}, \tilde{K}^{(4+)})$ are quartic and higher terms.
The cubic non-resonant terms can be eliminated by a further cubic normal form transformation, so there exist normal form variables $(\tilde{\tilde{W}}, \tilde{\tilde{Q}}): = (\tilde{W}, \tilde{Q}) + (W_{[3]}, Q_{[3]})$ that solve the system
\begin{equation}\label{NF4}
\begin{cases}		\tilde{\tilde{W}}_t+\tilde{\tilde{Q}}_{\al}=\tilde{G}_r^{(3)}+\tilde{G}_{null}^{(3)}+ \tilde{\tilde{G}}^{(4+)},\\
\tilde{\tilde{Q}}_t-i\tilde{\tilde{W}}_{\al\al\al\al}= \tilde{K}_r^{(3)}+  \tilde{K}_{null}^{(3)}+ \tilde{\tilde{K}}^{(4+)},
		\end{cases}
\end{equation}
where $(W_{[3]}, Q_{[3]})$ are cubic normal forms.
The null terms vanish when applied to three linear waves for \eqref{ZeroLinear} which travel in the same direction.
Hence, only resonant terms matter for our analysis.

By direct computation in \cite{MR4846707}, the cubic resonant terms are given by
\begin{equation*}
\begin{cases}
\tilde{G}_{r}^{(3)}&=\f{i}{128}\p_\al^{-2}(K^{(2),h}\bar{Q})=-\f{i}{128}\p_\al^{-2}\left(Q_{\al}^2\bar{Q}+\f52iW_\al W_{\al\al\al\al}\bar{Q}+5iW_{\al\al}W_{\al\al\al}\bar{Q}\right), \\
\tilde{K}_r^{(3)}&=\f{49}{32}\bar{W}_{\al}K^{(2),h}+W_{\al}|Q_{\al}|^2
+i\left[\f74|W_{\al}|^2W_{\al\al\al\al}+\f{15}{2}\bar{W}_{\al}W_{\al\al}W_{\al\al\al}-\f{15}{8}W^2_{\al}\bar{W}_{\al\al\al\al}\right]\\
&=-\f{49}{32}\bar{W}_{\al}Q_{\al}^2+W_{\al}|Q_{\al}|^2-i\left(\f{133}{64}|W_{\al}|^2W_{\al\al\al\al}+\f{5}{32} \bar{W}_{\al}W_{\al\al}W_{\al\al\al}+\f{15}{8}W^2_{\al}\bar{W}_{\al\al\al\al}\right).
\end{cases}
\end{equation*}

Choosing the diagonal variables  
\begin{equation*}
Y^- = \f12(|D|^\f32\tilde{\tilde{W}}- \tilde{\tilde{Q}}), \quad Y^+ = \f12(|D|^\f32\tilde{\tilde{W}}+ \tilde{\tilde{Q}})
\end{equation*}
for the symmetrization, we obtain the approximate system 
\begin{equation} \label{YPMEqn}
\begin{cases}
(i\partial_t-|D|^\f52)Y^- \approx \frac{i}{2}|D|^{\f32}\tilde{G}_r^{(3)}- \frac{i}{2}\tilde{K}_r^{(3)}, \\
(i\partial_t+|D|^\f52)Y^+ \approx \frac{i}{2}|D|^{\f32}\tilde{G}_r^{(3)}+ \frac{i}{2}\tilde{K}_r^{(3)}.
\end{cases}
\end{equation}

As in the heuristic discussion at the beginning of this section, we are interested in solutions that are localized near frequency $\xi_0 = -1$ along the branch $\tau = - |\xi|^\f52$.
Hence, we discard the second equation in \eqref{YPMEqn} and set $\p_\al = \pm i$ ($+$ for holomorphic variables and $-$ for anti-holomorphic variables), and $|D| = 1$ on the first equation of \eqref{YPMEqn} so that this corresponds to $\xi_0 = -1$.

Writing 
\begin{equation*}
   |D|^{\f32} \tilde{\tilde{W}} = Y^- + Y^+, \quad \tilde{\tilde{Q}} = Y^+ - Y^-.
\end{equation*}
Setting $Y^+ = 0$ and substituting above relations in the first equation of \eqref{YPMEqn}, we obtain an approximate equation for $Y^-$:
\begin{equation} \label{YMinusEqn}
 (i\partial_t-|D|^\f52)Y^- \approx  \f{57}{512} Y^- |Y^-|^2.
\end{equation}

We now begin with $\tilde{Y}^\epsilon$, which is the solution of the equation \eqref{TildeYEpsilon} with $\lambda = \frac{57}{512}$, and construct approximate solutions $(W^\epsilon, Q^\epsilon)$.
Owing to the equation \eqref{YMinusEqn}, we first choose
\begin{equation*}
    Y^- = \tilde{Y}^\epsilon.
\end{equation*}

We then exploit the second equation in \eqref{YPMEqn} to choose $Y^+$.
Indeed, the second equation in \eqref{TildeYEpsilon} can be rewritten as
\begin{equation*}
  (i\partial_t+|D|^\f52)Y^+ \approx  -\f{31}{512} Y^- |Y^-|^2.  
\end{equation*}
Since $(i\partial_t + |D|^{\f52}) \approx 2$ near $\xi_0 = 1$ along $\tau = - |\xi|^{\f52}$, we then choose
\begin{equation*}
    Y^+ = -\frac{31}{1024} \tilde{Y}^\epsilon |\tilde{Y}^\epsilon|^2.
\end{equation*}
The above choices of $Y^\pm$ lead to the approximate normal form variables $\tilde{\tilde{W}}^\epsilon$ and $\tilde{\tilde{Q}}^\epsilon$.
They are given by
\begin{equation} \label{DoubleTildeWQEpsilon}
 \tilde{\tilde{W}}^\epsilon = |D|^{-\f32}\left(\tilde{Y}^\epsilon -\frac{31}{1024} \tilde{Y}^\epsilon |\tilde{Y}^\epsilon|^2\right), \quad \tilde{\tilde{Q}}^\epsilon = -\tilde{Y}^\epsilon-\frac{31}{1024} \tilde{Y}^\epsilon |\tilde{Y}^\epsilon|^2.
\end{equation}

 $(\tilde{\tilde{W}}^\epsilon, \tilde{\tilde{Q}}^\epsilon)$ that we obtained can be viewed as the approximate solutions of the system \eqref{NF4}.
 To obtain the approximate solution of the system \eqref{NF1}, we invert the quadratic and cubic normal form corrections up to cubic terms.
Hence, we get
\begin{equation} \label{WQeDef}
\begin{cases}
W^{\e}=&\tilde{\tilde{W}}^\epsilon-\tilde{\tilde{W}}^{\e}_{[2]}-B^h(\tilde{\tilde{W}}^{\e}_{[2]},\tilde{\tilde{W}}^\epsilon)-B^h(\tilde{\tilde{W}}^{\e},\tilde{\tilde{W}}^\epsilon_{[2]})-C^h(\tilde{\tilde{Q}}^\epsilon_{[2]},\tilde{\tilde{Q}}^\epsilon)
-C^h(\tilde{\tilde{Q}}^\epsilon,\tilde{\tilde{Q}}^\epsilon_{[2]})\\
&-B^a(\tilde{\tilde{W}}^\epsilon_{[2]},\overline{\tilde{\tilde{W}}^\epsilon})-B^a(\tilde{\tilde{W}}^\epsilon,\overline{\tilde{\tilde{W}}^\epsilon_{[2]}})-C^a(\tilde{\tilde{Q}}^\epsilon_{[2]},\overline{\tilde{\tilde{Q}}^\epsilon})-C^a(\tilde{\tilde{Q}}^\epsilon,\overline{\tilde{\tilde{Q}}^\epsilon_{[2]}})-\tilde{\tilde{W}}^\epsilon_{[3]},\\
Q^{\e}=&\tilde{\tilde{Q}}^\epsilon -\tilde{\tilde{Q}}^{\e}_{[2]} -A^h(\tilde{\tilde{W}}^\epsilon_{[2]},\tilde{\tilde{Q}}^\epsilon) - A^h(\tilde{\tilde{W}}^\epsilon,\tilde{\tilde{Q}}^\epsilon_{[2]}) - A^a(\tilde{\tilde{W}}^\epsilon_{[2]},\overline{\tilde{\tilde{Q}}^\epsilon}) - A^a(\tilde{\tilde{W}}^\epsilon,\overline{\tilde{\tilde{Q}}^\epsilon_{[2]}}) \\
& - D^a(\tilde{\tilde{Q}}^\epsilon_{[2]},\overline{\tilde{\tilde{W}}^\epsilon}) - D^a(\tilde{\tilde{Q}}^\epsilon,\overline{\tilde{\tilde{W}}^\epsilon_{[2]}}) -\tilde{\tilde{Q}}^\epsilon_{[3]},
\end{cases}
\end{equation}
where approximate quadratic normal form corrections are
\begin{align*}
\tilde{\tilde{W}}^{\e}_{[2]}:=&B^h(\tilde{\tilde{W}}^\epsilon,\tilde{\tilde{W}}^\epsilon)+C^h(\tilde{\tilde{Q}}^\epsilon,\tilde{\tilde{Q}}^\epsilon)+B^a(\tilde{\tilde{W}}^\epsilon,\overline{\tilde{\tilde{W}}^\epsilon})+C^a(\tilde{\tilde{Q}}^\epsilon,\overline{\tilde{\tilde{Q}}^\epsilon}),\\
\tilde{\tilde{Q}}^{\e}_{[2]}:=& A^h(\tilde{\tilde{W}}^\epsilon,\tilde{\tilde{Q}}^\epsilon)+A^a(\tilde{\tilde{W}}^\epsilon,\overline{\tilde{\tilde{Q}}^\epsilon})+D^a(\tilde{\tilde{Q}}^\epsilon,\overline{\tilde{\tilde{W}}^\epsilon}).
\end{align*}
For approximate cubic normal form corrections $(\tilde{\tilde{W}}^\epsilon_{[3]}, \tilde{\tilde{Q}}^\epsilon_{[3]})$ above, we also replace $(W, Q)$ by $(\tilde{\tilde{W}}, \tilde{\tilde{Q}})$ in the original cubic normal form $(W_{[3]}, Q_{[3]})$.

To show that $(W^\epsilon, Q^\epsilon)$ are good approximation of the hydroelastic waves, we proceed in two steps.
We first show that the intermediate approximate solutions 
 $(\tilde{\tilde{W}}^\epsilon, \tilde{\tilde{Q}}^\epsilon)$ solve the system
 \begin{equation}\label{NF5}
\begin{cases}		\tilde{\tilde{W}}^\epsilon_t+\tilde{\tilde{Q}}^\epsilon_{\al}=\tilde{G}_r^{(3)}(\tilde{\tilde{W}}^\epsilon, \tilde{\tilde{Q}}^\epsilon)+\tilde{G}_{null}^{(3)}(\tilde{\tilde{W}}^\epsilon, \tilde{\tilde{Q}}^\epsilon)+ \tilde{g}^{\epsilon}\,,\\
\tilde{\tilde{Q}}^\epsilon_t-i\tilde{\tilde{W}}^\epsilon_{\al\al\al\al}= \tilde{K}_r^{(3)}(\tilde{\tilde{W}}^\epsilon, \tilde{\tilde{Q}}^\epsilon)+  \tilde{K}_{null}^{(3)}(\tilde{\tilde{W}}^\epsilon, \tilde{\tilde{Q}}^\epsilon)+ \tilde{k}^{\epsilon}\,,
\end{cases}
\end{equation}
where cubic source terms $(\tilde{G}_r^{(3)}, \tilde{K}_r^{(3)})$ and $(\tilde{G}_{null}^{(3)}, \tilde{K}_{null}^{(3)})$ are cubic resonant and null terms that are defined in \eqref{NF4}.
Note that $(\tilde{\tilde{W}}^\epsilon, \tilde{\tilde{Q}}^\epsilon)$ are frequency localized at $\xi \in (-2, 0)$. 
More precisely, we show that they have the following properties.
\begin{lemma}
For the approximate functions $(\tilde{\tilde{W}}^\epsilon, \tilde{\tilde{Q}}^\epsilon)$ that are defined in \eqref{DoubleTildeWQEpsilon}, they are close to $(\tilde{Y}^\epsilon, - \tilde{Y}^\epsilon)$ such that
\begin{equation} \label{DoulbleTildeWQYdiff}
 \|\tilde{\tilde{W}}^\epsilon - \tilde{Y}^\epsilon\|_{L^2} + \|\tilde{\tilde{Q}}^\epsilon + \tilde{Y}^\epsilon\|_{L^2} \lesssim \epsilon^\f32.
\end{equation}
Moreover, they are good approximate solutions in the sense that error terms $(\tilde{g}^\epsilon, \tilde{k}^\epsilon)$ in \eqref{NF5} are small such that
\begin{equation} \label{TildegkEst}
 \|\tilde{g}^\epsilon\|_{L^2} + \|\tilde{k}^\epsilon\|_{L^2} \lesssim \epsilon^{\f72}. 
\end{equation}
\end{lemma}

\begin{proof}
Recall that $\tilde{Y}^\epsilon$ satisfies the bounds \eqref{TildeYEpsilonBound}. We then have
\begin{align*}
\|\tilde{\tilde{W}}^\epsilon - \tilde{Y}^\epsilon\|_{L^2} + \|\tilde{\tilde{Q}}^\epsilon + \tilde{Y}^\epsilon\|_{L^2}\lesssim& \|(|D|^{-\f32}-1)\tilde{Y}^\epsilon\|_{L^2}+\|\tilde{Y}^\epsilon\|_{L^2}\|\tilde{Y}^\epsilon\|^2_{L^\infty}\\
\lesssim& \|(D+1)\tilde{Y}^\epsilon\|_{L^2}+\|\tilde{Y}^\epsilon\|_{L^2}\|\tilde{Y}^\epsilon\|^2_{L^\infty}\lesssim\e^{\f32}.
\end{align*}
It follows from the bounds \eqref{TildeYEpsilonBound} that we can freely replace $\tilde{\tilde{W}}^\e$ and $\tilde{\tilde{Q}}^\e$ by $\tilde{Y}^\e$ in cubic source terms in \eqref{NF5} up to negligible errors. From this, we can estimate directly the null parts  
\[
\| \tilde{G}_{\text{null}}^{(3)} \|_{L^2} + \| \tilde{K}_{\text{null}}^{(3)} \|_{L^2} \leq \epsilon^{\frac{7}{2}}.
\]

Next, we consider the time derivatives of $\tilde{\tilde{W}}^\e$ and $\tilde{\tilde{Q}}^\e$. For $\tilde{\tilde{Q}}^\e$ we have  
\[
\partial_t \tilde{\tilde{Q}}^\e = -\partial_t \tilde{Y}^\e - \frac{31}{1024} \left(2\partial_t \tilde{Y}^\e |\tilde{Y}^\e|^2 + \partial_t \bar{\tilde{Y}}^\e (\tilde{Y}^\e)^2\right).
\]

For $\tilde{Y}^\e$, it follows from \eqref{TildeYEpsilon} that
\[
\partial_t \tilde{Y}^\e = i\omega_0 (D) \tilde{Y}^\e -i\frac{57}{512} \tilde{Y}^\e|\tilde{Y}^\e|^2 - ig^\e.
\]

Using again \eqref{TildeYEpsilonBound}, we can discard the quintic terms arising from $\partial_t \tilde{Y}^\e$. Since
\begin{align*}
\|(D+1)\tilde{Y}^{\e}\|_{L^2}\lesssim \epsilon^\f32,
\end{align*}
we can replace $i\omega_0 (D) \tilde{Y}^\e$ by $-i \tilde{Y}^\e$ in the cubic terms, to obtain 
\[
\partial_t \tilde{\tilde{Q}}^\e = -i\omega_0 (D) \tilde{Y}^\e + i\frac{145}{1024} \tilde{Y}^\e |\tilde{Y}^\e|^2 + O_{L^2} (\epsilon^{\frac{7}{2}}).
\]

Similarly, we have
\begin{align*}
&\partial_{\al} \tilde{\tilde{Q}}^\e = i |D|^{\frac{1}{2}}\tilde{Y}^\e +i\frac{31}{1024} \tilde{Y}^\e |\tilde{Y}^\e|^2 + O_{L^2} (\epsilon^{\frac{7}{2}}), \\
&\partial_t \tilde{\tilde{W}}^\e = i |D|^{-\frac{3}{2}} \omega_0 (D) \tilde{Y}^\e-i\f{83}{1024}\tilde{Y}^\e |\tilde{Y}^\e|^2 + O_{L^2} (\epsilon^{\frac{7}{2}}), \\
& \partial_\al^4 \tilde{\tilde{W}}^\e = |D|^{\f52} \tilde{Y}^\e - \frac{31}{1024} \tilde{Y}^\e |\tilde{Y}^\e|^2 + O_{L^2} (\epsilon^{\frac{7}{2}}).
\end{align*}

Thus, we compute using the facts that $\omega_0(D)$ agrees to $-|D|^{\f52}$ in cubic order at $\xi = -1$ and $\xi$ is localized near $-1$, we get
\begin{align*}
e_1 :&= \partial_t \tilde{\tilde{W}}^\e+ \partial_\al\tilde{\tilde{Q}}^\e - \tilde{G}_{r}^{(3)} (\tilde{Y}^\e, \tilde{Y}^\e) = i|D|^{-\f32}(\omega_0 (D) +|D|^{\frac{5}{2}}) \tilde{Y}^\e + O_{L^2} (\epsilon^{\frac{7}{2}}) = O_{L^2} (\epsilon^{\frac{7}{2}}), \\
e_2 :&= \partial_t \tilde{\tilde{Q}}^\e-i \tilde{\tilde{W}}^\e_{\al \al \al \al} - \tilde{K}_{r}^{(3)} (\tilde{Y}^\e, \tilde{Y}^\e) = -i (\omega_0 (D) +|D|^{\frac{5}{2}}) \tilde{Y}^\e + O_{L^2} (\epsilon^{\frac{7}{2}}) = O_{L^2} (\epsilon^{\frac{7}{2}}).
\end{align*}
Putting $e_1$ and $\tilde{G}^{(3)}_{null}$ into $\tilde{g}^\epsilon$ and $e_2$ and $\tilde{K}^{(3)}_{null}$ into $\tilde{k}^\epsilon$, we obtain the desired estimate.
\end{proof}

Next, we consider the approximate solutions $(W^\epsilon, Q^\epsilon)$.
For source terms $(G, K)$ in \eqref{NF1}, we show that $(W^\epsilon, Q^\epsilon)$ solve the following equations:
\begin{equation}\label{NF6}
\begin{cases}		W^\epsilon_t+Q^\epsilon_{\al}=G^{(2)}(W^\epsilon, Q^\epsilon)+G^{(3)}(W^\epsilon, Q^\epsilon)+ g^\epsilon\,,\\
Q^\epsilon_t-iW^\epsilon_{\al\al\al\al}=K^{(2)}(W^\epsilon, Q^\epsilon)+K^{(3)}(W^\epsilon, Q^\epsilon) +k^\epsilon\,.
\end{cases}
\end{equation}

\begin{lemma}
For the approximate functions $(W^\epsilon, Q^\epsilon)$ that are defined in \eqref{WQeDef}, they are close to $(\tilde{Y}^\epsilon, - \tilde{Y}^\epsilon)$ such that
\begin{equation} \label{WQYEpsilondiff}
 \|W^\epsilon - \tilde{Y}^\epsilon\|_{L^2} + \|Q^\epsilon + \tilde{Y}^\epsilon\|_{L^2} \lesssim \epsilon^\f32.
\end{equation}
Moreover, they are good approximate solutions in the sense that error terms $(g^\epsilon,k^\epsilon)$ in \eqref{NF6} are small such that
\begin{equation*}
 \|g^\epsilon\|_{L^2} + \|k^\epsilon\|_{L^2} \lesssim \epsilon^{\f72}. 
\end{equation*}
\end{lemma}

\begin{proof}
Using estimates \eqref{TildeYEpsilonBound} and \eqref{DoulbleTildeWQYdiff}, we get 
\begin{equation} \label{DoubleTildeWQEst}
\|\tilde{\tilde{W}}^\epsilon \|_{L^2} + \|\tilde{\tilde{Q}}^\epsilon \|_{L^2}\lesssim \epsilon^\f12, \quad \|\tilde{\tilde{W}}^\epsilon \|_{L^\infty} + \|\tilde{\tilde{Q}}^\epsilon \|_{L^\infty} \lesssim \epsilon.
\end{equation}
The definition of $(W^\e, Q^\e)$ in \eqref{WQeDef} shows
\begin{equation*}
 \|W^\epsilon - \tilde{\tilde{W}}^\epsilon\|_{L^2} + \|Q^\epsilon - \tilde{\tilde{Q}}^\epsilon\|_{L^2} \lesssim \epsilon^\f32.  
\end{equation*}
This gives the estimate \eqref{WQYEpsilondiff} by combining with the estimate \eqref{DoulbleTildeWQYdiff}.
According to the definition of $(W^\epsilon, Q^\epsilon)$ in \eqref{WQeDef}, they are formally the inverse normal form transformations up of $(\tilde{\tilde{W}}, \tilde{\tilde{Q}})$ to cubic terms. 
Hence, error terms $(g^\epsilon, k^\epsilon)$ contain either quadratic and higher order terms in $(\tilde{\tilde{W}}^\e, \tilde{\tilde{Q}}^\e)$ or bilinear and higher order terms in $(\tilde{\tilde{W}}^\e, \tilde{\tilde{Q}}^\e)$ and $(\tilde{g}^\e, \tilde{k}^\e)$.
In either case, by using \eqref{DoubleTildeWQEst} and \eqref{TildegkEst}, one obtains the desired estimate.
\end{proof}

Finally in this subsection, we take into account the quartic and higher terms in the hydroelastic waves.
These terms satisfy
\begin{equation*}
 \|(G^{(\geq 4)}(W^\e, Q^\e), K^{(\geq 4)}(W^\e, Q^\e))\|_{\H} \lesssim \epsilon^{\f72}.
\end{equation*}
We then conclude the following NLS approximation result for the approximate solutions $(W^\e, Q^\e)$ of the hydroelastic waves.
\begin{proposition} \label{t:ApproximateWQ}
Let $\epsilon>0$, and $\tilde{Y}^\e$ be the approximate solution to the equation \eqref{TildeYEpsilon}.
Then there exists a frequency localized approximate solution $(W^\e, Q^\e)$ for the hydroelastic waves \eqref{HF14} with the following properties:
\begin{enumerate}
\item $(W^\e, Q^\e)$ are close to $(\tilde{Y}^\e, -\tilde{Y}^\e)$ in the sense of \eqref{WQYEpsilondiff}.
\item $(W^\e, Q^\e)$ solve the system
\begin{equation*}
\begin{cases}
&W^\e_t+F^\e(1+W^\e_{\al})= g^\e,\\
&Q^\e_t+F^\e Q^\e_{\al}+\nP\left[\f{|Q^\e_{\al}|^2}{J_\e}\right]
-i \nP\left\{\f{1}{J_\e^{\f12}}\f{d}{d\al}\left[\f{1}{J_\e^{\f12}}\f{d}{d\al}\left(\f{W^\e_{\al\al}}{J_\e^{\f12}(1+W^\e_{\al})}
-\f{\bar{W}^\e_{\al\al}}{J_\e^{\f12}(1+\bar{W}^\e_{\al})}\right)\right]\right\}\\
&-\f{i}{2} \nP\left\{\left[\f{W^\e_{\al\al}}{J_\e^{\f12}(1+W^\e_{\al})}-\f{\bar{W}^\e_{\al\al}}{J_\e^{\f12}(1+\bar{W}^\e_{\al})}\right]^3\right\}= k^\e,
\end{cases}
\end{equation*}
where the error terms on the right-hand side satisfy
\begin{equation*}
 \|(g^\e, k^\e)\|_{\H} \lesssim \e^{\f72}.
\end{equation*}
\end{enumerate}
\end{proposition}

\subsection{Replacing approximate solutions by exact solutions} \label{s:ReplaceExact}
Having constructed the approximate hydroelastic wave solutions $(W^\e, Q^\e)$, we now compare them with the exact solution, so that the conclusion of Proposition \ref{t:ApproximateWQ} also holds for $(W, Q)$.
To estimate the difference between $(W^\e, Q^\e)$ and $(W, Q)$, we need the modified energy estimate for the linearized hydroelastic wave system.

We consider a two-parameter family of functions $(W(t,s), Q(t,s))$ for $(t,s) \in [0, T\e^{-2}]\times [0, T\e^{-2}]$ that solve the hydroelastic waves \eqref{HF14} with the initial condition at $t = s$:
\begin{equation*}
    W(s,s) = W^\e(s), \quad Q(s,s) = Q^\e(s).
\end{equation*}
Here $t$ is the time variable, and $s$ can be viewed as a parameter.

We first claim that for sufficiently small $\epsilon$, the solutions exist up to time $T\epsilon^{-2}$ and, moreover, that the control parameters $A_0(t,s)$, $A_{\f52}(t,s)$ $\mathcal{A}_{2}(t,s)$ defined in \eqref{MRA} and \eqref{D:A74} satisfy the uniform bounds
\[A_0(t,s), A_{\f52}(t,s), \mathcal{A}_{2}(t, s)\leq M\epsilon\] for some universal constant $M$.

To prove this, we employ a continuity argument. 
Since $A_0, A_{\f52}$ and $\mathcal{A}_{2}$ are continuous functions of $t$ and $s$, we make the bootstrap assumptions
\[A_0(t,s), A_{\f52}(t,s),\mathcal{A}_{2} \leq 2M\epsilon, \quad s_0 \leq s \leq t \leq T\epsilon^{-2},\]
and then show that for a sufficiently large $M$, the constant $2M$ can be improved to $M$.
This improvement implies that the solutions can be extended up to $T\epsilon^{-2}$, and hence the desired bounds hold on the entire interval.
Using the modified energy estimate Theorem \ref{t:QMEHP}, we get
\begin{equation} \label{WREnergyBound}
 \|(\mathbf{W}, R)\|_{\H^N} \lesssim e^{4CM^2T}\epsilon^\f12.
\end{equation}

Note that $(\partial_s W(t,s), \partial_s Q(t,s))$ solve the linearized hydroelastic wave system \eqref{wqLinearized}. 
When $t = s$, we have the initial condition
\[
(w,q)(s,s) = (\tilde{g}^\e(s), \tilde{k}^\e(s)).
\]

Switching to the good variable $r = q - Rw$ and applying Theorem \ref{t:Hswellposedflow} yield the uniform bound of $(w,r)$,
\[
\|(w,r)(t,s)\|_{\H} \lesssim e^{4CM^2T} \epsilon^{\frac{7}{2}}.
\]

Integrating in $s$, the bound on $w$ gives
\begin{equation}\label{Ws-t}
\|W(t,s)-W(t,t)\|_{H^{\f32}} \lesssim T e^{4CM^2T} \epsilon^{\frac{3}{2}}.  
\end{equation}
Since $Q_s=r+Rw$, by the bounds of $(w, r)$, we obtain in a similar way that
\begin{equation}\label{Qs-t}
\|Q(t,s)-Q(t,t)\|_{L^2} \lesssim T e^{4CM^2T} \epsilon^{\frac{3}{2}}.  
\end{equation}

Next, we consider $R$, for which we have
\begin{equation*}
R_s=\f{r_{\al}+R_{\al}w}{1+W_{\al}}.
\end{equation*}
This implies
\begin{align*}
\|R_s\|_{H^{-1}}\lesssim e^{8CM^2T} \epsilon^{\frac{7}{2}}.
\end{align*}
Integrating in $s$ then yields the estimate
\begin{equation}\label{Rs-t}
\|R(t,s)-R(t,t)\|_{H^{-1}}\lesssim Te^{8CM^2T} \epsilon^{\frac{3}{2}}.
\end{equation}
Combining the bounds \eqref{WREnergyBound}, \eqref{Ws-t}, and \eqref{Rs-t} and interpolating, we obtain
\begin{equation*}
\|(\mathbf{W}(t,s)-\mathbf{W}(t,t),R(t,s)-R(t,t))\|_{\H^N} \lesssim Te^{8CM^2T} \epsilon^{\frac{3}{2}-\delta}.
\end{equation*}
Finally, using Sobolev embeddings, we obtain
\[\mathcal{A}_{2}(t,s)
\lesssim e^{8CM^2T} \epsilon^{\frac{3}{2}-\delta} + \mathcal{A}_{2}(t,t)\lesssim Te^{8CM^2T} \epsilon^{\frac{3}{2}-\delta} + \epsilon.
\]
All implicit constants here are independent of $M$. 
Similar estimates also hold for $A_0(t,s)$ and $A_{\f52}(t,s)$.
This completes the bootstrap argument, and finishes the proof of Theorem \ref{t:NLSApproximation}.

\section{The proof of nonlinear modulational instability of the Stokes waves} \label{s:Modulation}

In this section, we use the nonlinear modulational instability of the Stokes waves of the cubic NLS equation to deduce the nonlinear modulational instability of the Stokes waves of the hydroelastic waves.

Consider a 1-D cubic focusing NLS equation 
\begin{equation} \label{StandardCubicNLS}
    (i\p_t + \p_\al^2)u + |u|^2u = 0.
\end{equation}
If we set 
\begin{equation} \label{UuRelation}
U(t,\al)=i\sqrt{\f{512}{57}}u\left(t,\sqrt{\f{8}{15}}\al\right),   
\end{equation}
then $U$ solves the equation \eqref{LambdaNLS}.
We further define $Y^\e$ as in \eqref{YURelation}.
According to the result in Theorem \ref{t:NLSApproximation}, there exists a small constant $\e$ such that for each $0<\epsilon< \epsilon_0$, there exists a pair of solution $(W, Q)$ to the system \eqref{HF14} for $t$ in the time interval $[0, T\epsilon^{-2}]$ with the estimate \eqref{WQYEpsilonDiffBd}.

Following the discussion from Appendix \ref{s:NLSInstability}, we take the special solution $u_{ST} = e^{it}$ to the equation \eqref{StandardCubicNLS}.
Then according to our derivation in Section \ref{s:Approximation},
\begin{equation*}
  Y^\e_{ST}(t,x) = i\e \sqrt{512/57} e^{-i(t+x)}e^{i\e^2t}
\end{equation*}
is a periodic traveling wave solution to \eqref{CubicEpsilonDispersive}.
On one hand, it follows from Theorem \ref{t:NLSApproximation} that for $t\in[0, T\epsilon^{-2}]$, there exists a pair of solutions $(W^{(1)}, Q^{(1)})$ such that
\begin{align*}
\|(W^{(1)}-i\e \sqrt{512/57} e^{-i(t+\alpha)}e^{i\e^2t},Q^{(1)}+i\e \sqrt{512/57}e^{-i(t+\alpha)}e^{i\e^2t})\|_{\H(\e^{-1}\T)}\lesssim  \e^{\f32}.
\end{align*}
On the other hand, by the asymptotic expansion of the Stokes waves in Theorem \ref{t:Existence}, with $ \varepsilon = \sqrt{\frac{512}{57}}\epsilon$, 
\begin{equation} \label{512Constant}
 \|(W_{ST},Q_{ST})(t,\alpha ) - i\e \sqrt{512/57}(e^{-i(t+\alpha)}e^{i\e^2t}, - e^{-i(t+\alpha)}e^{i\e^2t})  \|_{\H(\epsilon^{-1}\T)}\lesssim   \e^{\f32}.  
\end{equation}

Hence, we get for $t\in [0, T\epsilon^{-2}]$,
\begin{equation*}
  \|(W_{ST},Q_{ST})(t,\alpha ) - (W^{(1)}, Q^{(1)})  \|_{\H(\epsilon^{-1}\T)}\lesssim \e^{\f32}.  
\end{equation*}
In particular, when $t=0$,  
\begin{equation*}
  \|(W_{ST},Q_{ST})(0, \cdot) - (W^{(1)}_0, Q^{(1)}_0)  \|_{\H(\epsilon^{-1}\T)}\lesssim \e^{\f32}.  
\end{equation*}

We now consider a perturbation of the special solution $u_{ST}$ in the form
\begin{equation*}
    u^{(2)}(t,x) = u_{ST}(t,x) + e^{it}w(t,x).
\end{equation*}
Then the perturbation $w(t,x)$ solves the equation \eqref{wNLSEqn}.
By considering initial data of the type \eqref{InitialDataw} with
\begin{equation*}
 \delta = \frac{\epsilon}{2}, \quad |\delta_j| = \frac{\delta}{2}, \quad |\eta_j|\ll \delta,
\end{equation*}
the initial data $w_0$ satisfies the estimate
\begin{equation*}
    \| w_0\|_{H^1(\T)} \leq \f34 \e.
\end{equation*}
Moreover, according to Theorem \ref{t:NLSInstability}, for $\mu = \f12$, and $T_0 = |\log \e|$, 
\begin{equation*}
    \| w(t)\|_{H^1(\T)} \lesssim  1, \quad t\in [0, T_0]; \qquad \| w(T_0)\|_{H^1(\T)} \geq \f18.
\end{equation*}
We further define $U^{(2)}$, $U^{(2), \e}$ and $Y^{(2), \e}$ according to \eqref{UuRelation} and \eqref{YURelation}, so that
\begin{equation*}
 Y^{(2), \e}(t,\al) = Y^{\e}_{ST}(t,\al) + i\e \sqrt{512/57} e^{-i(t+x)+i\e^2t} w(\e^2 t , \sqrt{8/15}\e\al + 5/2 \e^2 t).
\end{equation*}
We have estimates at time $t=0$ and $t = T_\e : = \e^{-2}|\log \e|$, 
\begin{equation*}
 \|Y^{(2), \e}(0,\cdot) -Y^{\e}_{ST}(0,\cdot) \|_{H^1(\e^{-1}\T)} \lesssim \e^{\f32}, \quad \|Y^{(2), \e}(T_\e,\cdot) -Y^{\e}_{ST}(T_\e,\cdot) \|_{H^1(\e^{-1}\T)} \gtrsim \e^{\f12}.
\end{equation*}
Using again Theorem \ref{t:NLSApproximation}, for $t\in[0, \epsilon^{-2}|\log \e|]$, there exists a pair of solution $(W^{(2)}, Q^{(2)})$ such that
\begin{align*}
\|(W^{(2)}, Q^{(2)}) - (Y^{(2), \e}, -Y^{(2), \e})\|_{\H(\e^{-1}\T)}\lesssim  \e^{\f32}.
\end{align*}
The solution $(W^{(2)}, Q^{(2)})$ are close to the Stokes wave solutions initially in the sense that
\begin{align*}
&\|(W^{(2)}_0, Q^{(2)}_0) -  (W_{ST},Q_{ST})(0,\cdot)\|_{\H(\e^{-1}\T)} \lesssim \|(W^{(2)}_0, Q^{(2)}_0) -  (Y^{(2), \e}, -Y^{(2), \e})(0,\cdot)\|_{\H(\e^{-1}\T)} \\
+& \|Y^{(2), \e}(0,\cdot) -Y^{\e}_{ST}(0,\cdot)\|_{H^1(\e^{-1}\T)} + \| (Y^{\e}_{ST}, -Y^{\e}_{ST})(0,\cdot)-(W^{(1)}_0, Q^{(1)}_0)\|_{\H(\e^{-1}\T)} \\
+&   \| (W^{(1)}_0, Q^{(1)}_0)-(W_{ST},Q_{ST})(0, \cdot)  \|_{\H(\epsilon^{-1}\T)}\lesssim \e^{\f32}.
\end{align*}
However, at time $t = T_\e = \e^{-2}|\log \e| $, $(W^{(2)}, Q^{(2)})$ are no longer close to the Stokes wave solutions, as
\begin{align*}
& \|(W^{(2)}, Q^{(2)})(T_{\e},\cdot) -  (W_{ST},Q_{ST})(T_{\e},\cdot )\|_{\H(\e^{-1}\T)} \\
\gtrsim& \|Y^{(2), \e}(T_\e,\cdot) -Y^{\e}_{ST}(T_\e,\cdot)\|_{H^1(\e^{-1}\T)}- \|(W^{(2)}, Q^{(2)})(T_\e, \cdot) -  (Y^{(2), \e}, -Y^{(2), \e})(T_\e,\cdot)\|_{\H(\e^{-1}\T)}  \\
&- \| (Y^{\e}_{ST}, -Y^{\e}_{ST})(T_\e,\cdot)-(W^{(1)}, Q^{(1)})(T_\e, \cdot)\|_{\H(\e^{-1}\T)}\\
&-   \| (W^{(1)}, Q^{(1)})(T_\e, \cdot)-(W_{ST},Q_{ST})(T_\e, \cdot)  \|_{\H(\epsilon^{-1}\T)}\gtrsim \e^{\f12}.
\end{align*}

Therefore, we show that the Stokes wave solutions $(W_{ST},Q_{ST})$ are nonlinearly modulationally unstable.

\appendix
\section{Local bifurcation theorem} \label{s:bifurcation}
This Section recalls the local bifurcation theorem by Crandall and Rabinowitz \cite{MR288640}.
\begin{theorem}[Local bifurcation, \cite{MR288640}]
Suppose that $X$ and $Y$ are Banach spaces, that $F : \mathbb{R} \times X \to Y$ is of class $C^k$, $k \geq 2$, and that $F(\lambda, 0) = 0 \in Y$ for all $\lambda \in \mathbb{R}$. Suppose also that
\begin{itemize}
\item $L = \partial_x F[(\lambda_0, 0)]$ is a Fredholm operator of index zero.

\item $\ker(L)$ is one-dimensional.
 $\ker(L) = \{\xi \in X : \xi = s\xi_0 \text{ for some } s \in \mathbb{R} \}$, $\xi_0 \in X \setminus \{0\}$.

\item The transversality condition holds:
\begin{equation*}
\partial_{\lambda, x}^2 F[(\lambda_0, 0)](1, \xi_0) \notin \text{range}(L). 
\end{equation*}
\end{itemize}
Then $(\lambda_0, 0)$ is a bifurcation point. More precisely, there exists $\epsilon > 0$ and a branch of solutions
\[
    \{(\lambda, x) = (\Lambda(s), s\chi(s)) : s \in \mathbb{R}, |s| < \epsilon \} \subset \mathbb{R} \times X,
\]
such that 
\begin{itemize}
\item $\Lambda(0) = \lambda_0$, $\chi(0) = \xi_0$.

\item $F(\Lambda(s), s\chi(s)) = 0$ for all $s$ with $|s| < \epsilon$.

\item $\Lambda$ and $s \mapsto s\chi(s)$ are of class $C^{k-1}$, and $\chi$ is of class $C^{k-2}$ on $(-\epsilon, \epsilon)$.

\item There exists an open set $U_0 \subset \mathbb{R} \times X$ such that $(\lambda_0, 0) \in U_0$ and
\begin{align*}
\{(\lambda, x) \in U_0 : F(\lambda, x) &= 0, x \neq 0\} = \{(\Lambda(s), s\chi(s)) : 0 < |s| < \epsilon\}.
    \end{align*}
\item If $F$ is analytic, $\chi$ and $\Lambda$ are analytic functions on $(-\epsilon, \epsilon)$.
\end{itemize}
\end{theorem}

To show that an operator is a Fredholm operator with index zero, one may use Theorem 2.7.6 from \cite{MR1956130}.
\begin{lemma}[\hspace{1sp}\cite{MR1956130}]\label{t:Fredholm}
Suppose $X$ and $Y$ are Banach spaces, $K \in \mathcal{L}(X, Y)$ is compact and $T \in \mathcal{L}(X, Y)$ is a homeomorphism. 
Then $B= T + K$ is Fredholm with index zero.
\end{lemma}

\section{Modified energy estimate for hydroelastic waves} \label{s:Modified}
In this section, we recall the modified energy estimates for hydroelastic waves from \cite{MR4846707}.
Then, we recall the linearized hydroelastic waves from \cite{WanY}, and obtain its modified energy estimate.

We first introduce the control norms
\begin{equation}\label{MRA}
A_0:=\|\bw\|_{L^{\infty}}+\|Y\|_{L^{\infty}}+\|D^{-\f32}R\|_{L^{\infty}\cap B^{0}_{\infty, 2}}, \quad A_{\f52}:=\||D|^{\f52}\bw\|_{L^{\infty}}+\|R_{\al}\|_{L^{\infty}}.
\end{equation}
Then, we have the following modified energy estimate.
\begin{theorem}[\hspace{1sp}\cite{MR4846707}]\label{t:QMEHP}
For any $n\geq5$, there exists an energy functional $E^{n,(3)}$ which has the following properties as long as $A_0\ll1$:\\
(i) Norm equivalence:
\[
E^{n,(3)}(\bw,R)=(1+O(A_0))E_0(\p^{n-1}\bw,\p^{n-1}R),
\]
(ii) Cubic energy estimates:
\[
\f{d}{dt}E^{n,(3)}(\bw,R)\lesssim_{A_0}A_0 A_{\f52}E^{n,(3)}(\bw,R).
\]
\end{theorem}

Let the solutions for the linearized hydroelastic wave equations around $(W,Q)$ to \eqref{HF14} be $(w,q)$.
Then according to the computation in \cite{WanY}, the linearized hydroelastic waves are as follows:
\begin{equation} \label{wqLinearized}
\begin{cases}
w_t+Fw_{\al}+(1+\mathbf{W})\nP[m-\bar{m}]=0,\\
q_t+Fq_{\al}+Q_{\al}\nP[m-\bar{m}]+\nP[n+\bar{n}]-i\nP[p-\bar{p}]=0,
\end{cases}
\end{equation}
where linearized auxiliary functions are given by
\begin{align*}
m:=&\f{q_{\al}-Rw_{\al}}{J}+\f{\bar{R}w_{\al}}{(1+\mathbf{W})^2}=\f{r_{\al}+R_{\al}w}{J}+\f{\bar{R}w_{\al}}{(1+\mathbf{W})^2}, \\
n:=&\bar{R}\delta R=\f{\bar{R}(q_{\al}-Rw_{\al})}{1+\mathbf{W}}=\frac{\bar{R}(r_{\al}+R_{\al}w)}{1+\mathbf{W}}, \\
p:=&\f{1}{J^{\f12}}\f{d}{d\al}\left[\f{1}{J^{\f12}}\f{d}{d\al}\left(\f{w_{\al\al}}{J^{\f12}(1+\mathbf{W})}-\left(\f{3\mathbf{W}_{\al}}{2J^{\f12}(1+\mathbf{W})^2}-\f{\bar{\mathbf{W}}_{\al}}{2J^{\f32}}\right)w_{\al}\right)\right]\\
&-\f{1}{J^{\f12}}\f{d}{d\al}\left[\frac{(1+\bar{\mathbf{W}})w_{\al}}{2J^{\f32}}\f{d}{d\al}\left(\f{\mathbf{W}_{\al}}{J^{\f12}(1+\mathbf{W})}-\f{\bar{\mathbf{W}}_{\al}}{J^{\f12}(1+\bar{\mathbf{W}})}
\right)\right]\\
&-\frac{(1+\bar{\mathbf{W}})w_{\al}}{2J^{\f32}}\f{d}{d\al}\left[\f{1}{J^{\f12}}\f{d}{d\al}\left(\f{\mathbf{W}_{\al}}{J^{\f12}(1+\mathbf{W})}-\f{\bar{\mathbf{W}}_{\al}}{J^{\f12}(1+\bar{\mathbf{W}})}
\right)\right]\\
&+\f32\left(\f{w_{\al\al}}{J^{\f12}(1+\mathbf{W})}-\left(\f{3\mathbf{W}_{\al}}{2J^{\f12}(1+\mathbf{W})^2}-\f{\bar{\mathbf{W}}_{\al}}{2J^{\f32}}\right)w_{\al}\right)\left[\f{\mathbf{W}_{\al}}{J^{\f12}(1+\mathbf{W})}
-\f{\bar{\mathbf{W}}_{\al}}{J^{\f12}(1+\bar{\mathbf{W}})}\right]^2.
\end{align*}
Switching to the diagonal linearized variable $r:=q-Rw$, the system \eqref{wqLinearized} can be rewritten as
\begin{equation} \label{linearizedeqn}
\begin{cases}
(\p_t+b\p_{\al})w+\f{1}{1+\bar{\mathbf{W}}}r_{\al}+\frac{R_{\al}}{1+\bar{\mathbf{W}}}w=\mathcal{G}_0(w,r),\\
(\p_t+b\p_{\al})r-i\frac{a}{1+\mathbf{W}}w-i\nP p-\frac{w}{1+\mathbf{W}}\p_{\al}\nP c=\mathcal{K}_0(w,r),
\end{cases}
\end{equation}
where on the right-hand side,
\begin{equation*}
\mathcal{G}_0(w,r)=(1+\mathbf{W})(\nP\bar{m}+\bar{\nP}m),\quad
\mathcal{K}_0(w,r)=\bar{\nP}n-\nP\bar{n}-i\nP\bar{p}.
\end{equation*}
The real-valued \textit{frequency-shift} $a$ and the \textit{advection velocity} $b$ are given by
\begin{equation}\label{HF17}
		a:=i\left(\bar{P}[\bar{R}R_{\al}]-P[R\bar{R}_{\al}]\right),\, \quad b:= \nP\left[ \frac{R}{1+\bar{\bw}}\right] + \bar{\nP}\left[ \frac{\bar{R}}{1+\bw}\right],
\end{equation}
and the auxiliary function $c$ is given by
\begin{equation*}
i\tilde{c}:=\f{\mathbf{W}_{\al}}{J^{\f12}(1+\mathbf{W})}
-\f{\bar{\mathbf{W}}_{\al}}{J^{\f12}(1+\bar{\mathbf{W}})}, \quad
ic:=iJ^{-\f12} \p_\al(J^{-\f12}\tilde{c}_\al)+\f12(i\tilde{c})^3.
\end{equation*}

We introduce the control norms:
\begin{equation}\label{D:A74}
\mathcal{A}_0 : = \|(\bw, R) \|_{C^{\epsilon}\times C^{- \f32+\epsilon}}, \quad \mathcal{A}_{2} : = \|(\bw, R) \|_{W^{2+\epsilon, \infty}\times W^{\f12+\epsilon,\infty}}.
\end{equation}
Then we have the following modified energy estimate.
\begin{theorem}
\label{t:Hswellposedflow}
Assume that $\mathcal{A}_\f32 \lesssim 1$ and $\mathcal{A}_{2} \in L^2_t([0,T])$.
Then the linearized hydroelastic system \eqref{linearizedeqn} is well-posed in $\H$ on $[0,T]$.
Furthermore, there exists an energy functional $E^{0}_{lin}(w,r)$ satisfying the following properties on $[0,T]$:
\begin{enumerate}
    \item Norm equivalence:
    \begin{equation*}
        E^{0}_{lin}(w,r) = \left(1+O\left(\mathcal{A}_\f32 \right) \right) \| (w,r)\|^2_{\H}.
    \end{equation*}
    \item The time derivative of $E_{lin}^{0}(w,r)$ is bounded by:
    \begin{equation*}
        \frac{d}{dt} E_{lin}^{0}(w,r) \lesssim_{\CalAZ}  \mathcal{A}_{2}^2 E^{0}_{lin}(w,r).
    \end{equation*}
\end{enumerate} 
\end{theorem} 

We remark that in \cite{WanY}, authors proved the above result with control norm $\mathcal{A}_2$ replaced by the $L^4$ based control norm
\begin{equation*}
 \mathcal{A}_{\sharp,\f74} : = \|(\bw, R) \|_{W^{2+\epsilon,4}\times W^{\f12+\epsilon,4}}.
\end{equation*}
However, the control norm $\mathcal{A}_{\sharp,\f74}$ does not fit the analysis in Section \ref{s:ReplaceExact}.

Note that in \cite{WanY}, one of the key estimate that leads to the choice of the control norm $\mathcal{A}_{\sharp,\f74}$ is the estimate for the frequency shift $a$:
\begin{equation*} 
 \|a\|_{H^\epsilon} \lesssim \|R \|^2_{W^{\frac{1}{2}+\epsilon, 4}} \lesssim \ASSharp.  
\end{equation*}
One can replace this estimate by 
\begin{equation*}
\|a\|_{W^{\epsilon, \infty}} \lesssim \|R \|^2_{W^{\frac{1}{2}+\epsilon, \infty}} \lesssim \mathcal{A}_2^2.
\end{equation*}
By using $\mathcal{A}_2$ as the control norm and doing the rest estimates as in \cite{WanY}, we can prove Theorem \ref{t:Hswellposedflow}.

\section{Instability of cubic NLS equation} \label{s:NLSInstability}

In this section, we recall a constructive proof from Appendix D of \cite{MR4627323} on the instability of the Stokes wave of the focusing cubic NLS problem.

For the 1-D cubic focusing NLS equation \eqref{StandardCubicNLS}, it has a Stokes wave solution
\begin{equation*} %\label{StokesNLS}
    u(t,x) = e^{it}.
\end{equation*}
We consider the following perturbation of the Stokes waves
\begin{equation*}
    u(t,x) = e^{it}(1+ w(t,x)), \quad x\in q\T
\end{equation*}
for some positive real number $q$.
Substituting this expression in \eqref{StandardCubicNLS}, then according to the computation in Appendix D of  \cite{MR4627323}, $w$ solves the  equation
\begin{equation} \label{wNLSEqn}
(i\p_t + \p_x^2)w + 2\Re w = 0.
\end{equation}
We choose $k_0 \in \Z^+$ such that
\begin{equation*}
\left|\frac{k_0}{q}\right| \sqrt{2-\left|\frac{k_0}{q}\right|^2} = \tau,\quad\text{where } \tau = \sup_{k\in \Z} \Re \left|\frac{k}{q}\right| \sqrt{2-\left|\frac{k}{q}\right|^2}.
\end{equation*}
Given $0<\delta\ll 1$ and $s' > \f12$ fixed, we consider the initial data
\begin{equation} \label{InitialDataw}
  w_0 = \frac{1}{\sqrt{q}}\left(\delta_1 e^{i\frac{k_0}{q}x} + \delta_2 e^{-i\frac{k_0}{q}x} +\eta_1 e^{i\frac{x}{q}} +\eta_2 e^{-i\frac{x}{q}}\right),
\end{equation}
where $|\delta_j| = \frac{\delta}{2s'}\ll 1$, $|\eta_j|\ll |\delta_i|$.
Then $w_0$ satisfies the estimate
\begin{equation*}
 \|w_0\|_{H^{s'}(q\T)} \leq \f32 \delta.
\end{equation*}
The result in \cite{MR4627323} shows that the cubic NLS equation \eqref{StandardCubicNLS} is unstable near the Stokes wave solution $u(t,x)$ in the sense that the perturbation $w(t,x)$ does not remain small for all time.
More precisely, we have the following result.
\begin{theorem}\label{t:NLSInstability}
Consider the equation \eqref{wNLSEqn} with initial data \eqref{InitialDataw}.
Then there exist $\mu$ satisfying $|\delta| \ll \mu < 1$ and $T_0 = \log \left(\frac{\mu}{\delta}\right)$
such that
\begin{equation} \label{wT0Est}
 \|w(t) \|_{H^{s'}(q\T)} \lesssim \mu, \quad \forall t \in [0, T_0]; \qquad  \|w(T_0) \|_{L^2(q\T)} \geq \frac{\mu}{4} \gg \delta.
\end{equation}
\end{theorem}

\textbf{Acknowledgments.}
Jiaqi Yang is supported by National Natural Science Foundation of China under Grant: 12471225.
Authors would like to thank Prof. Baoping Liu for introducing this research topic. 

\bibliography{HWW}
\bibliographystyle{plain}
	
\end{document}